\documentclass[11pt]{amsart}
\usepackage{lmodern}
\usepackage[T1]{fontenc}
\usepackage{microtype}
\usepackage{amssymb}
\usepackage{amsthm}
\usepackage{amscd}
\usepackage{hyperref}
\usepackage{cite}
\usepackage{color}

\newtheorem{theorem}{Theorem}[section]

\theoremstyle{definition}
\newtheorem{definition}[theorem]{Definition}

\makeatletter

\def\dotminussym#1#2{%
  \setbox0=\hbox{$\m@th#1-$}%
  \kern.5\wd0%
  \hbox to 0pt{\hss\hbox{$\m@th#1-$}\hss}%
  \raise.6\ht0\hbox to 0pt{\hss$\m@th#1.$\hss}%
  \kern.5\wd0}

\mathchardef\mhyphen="2D

\allowdisplaybreaks[2]

\begin{document}

\title{Computable Ramsey's Theorem for Pairs Needs Infinitely Many $\Pi^0_2$ Sets}
\author{Gregory Igusa and Henry Towsner}
\date{\today}

\begin{abstract}

In \cite{J}, Theorem 4.2, Jockusch proves that for any computable k-coloring of pairs of integers, there is an infinite $\Pi^0_2$ homogeneous set. The proof uses a countable collection of $\Pi^0_2$ sets as potential infinite homogeneous sets. In a remark preceding the proof, Jockusch states without proof that it can be shown that there is no computable way to prove this result with a finite number of $\Pi^0_2$ sets. We provide a proof of this latter fact.

\end{abstract}


\maketitle

\section{Introduction}


In \cite{J}, Jockusch initiated the study of the effective content of Ramsey's theorem, stated below. This study has proved to be enormously fruitful in effective combinatorics, and also in reverse mathematics. In Theorem 4.2 of \cite{J}, Jockusch proves that for any computable k-coloring of pairs of integers, there is an infinite $\Pi^0_2$ homogeneous set. Before this proof, he makes the remark that even for 2-colorings of pairs of integers (basic recursive partitions, in his language), it can be shown that there is no uniform computable way to take an index for an arbitrary computable coloring, and to produce a finite number of indices of $\Pi^0_2$ sets with the property that one of those $\Pi^0_2$ sets will be an infinite homogeneous set for that coloring.

The proof of this fact appears to have been lost, and recently Jockusch has asked for a proof, which we present here.

\section{Definitions}

\begin{definition}

A $k$-coloring of the $n$-element subsets of $\mathbb{N}$ is a function $c:[\mathbb{N}]^n\rightarrow k$, from the set of unordered $n$-element subsets of $\mathbb{N}$, to $k$.

\end{definition}

We think of such a coloring as a rule that assigns a color to every $n$-element subset of $\mathbb{N}$, using up to $k$ different colors.

Ramsey's theorem is then the following theorem of combinatorics.

\begin{theorem}[Ramsey's Theorem]

For any $n,k\geq 1$, and any $k$-coloring $c:[\mathbb{N}]^n\rightarrow k$, there exists an infinite subset $H\subseteq\mathbb{N}$ such that $c\upharpoonright [H]^n$ is a constant function.
\end{theorem}

We call such an $H$ an infinite homogeneous set for $c$.

In this paper, we will be primarily concerned with the case when $n=k=2$.

In \cite{J}, Jockusch proves the following.

\begin{theorem}[Jockusch, \cite{J}, Theorem 4.2]

If $c:[\mathbb{N}]^2\rightarrow k$ is a computable k-coloring of pairs, then there exists an infinite $\Pi^0_2$ homogeneous set for $c$.
\end{theorem}

We prove the following.

\begin{theorem}
There does not exist a partial computable $f$ with the property that for any $e$, if $e$ is the code of a total computable 2-coloring $c:[\mathbb{N}]^2\rightarrow 2$ of pairs, then $f(e)$ halts, producing the code for a finite set $\{a_0,a_1,\dots,a_k\}$ of indices for $\Pi^0_2$ sets with the property that at least one of those $\Pi^0_2$ sets is an infinite homogeneous subset for $c$.

\end{theorem}

Indeed, we prove slightly more: that there is no such $f$ where $f(e)$ is the code for a c.e. set $W_{f(e)}$ enumerating finitely many codes for $\Pi^0_2$ sets.

\section{Trains}

\begin{definition}
  An \emph{$n$-train} is a sequence of distinct sets of size $n$, $\tau_0,\tau_1,\ldots,\tau_m$ such that for every $a\in\tau_{i+1}\setminus\tau_i$, $a>\tau_i$.

For instance
\[\{1,2,3\}, \{2,3,5\}, \{5, 7, 9\}, \{5,9,12\}\]
is a $3$-train.

If $0\leq k< n$ we write $\tau(k)$ for the $k+1$-st element of $\tau$ in the usual ordering of $\mathbb{N}$.
\end{definition}

Our main tool is the following combinatorial lemma.  We color pairs from $R,B$, and if $\iota\in\{R,B\}$, we write $\overline{\iota}$ for the opposite color: $\overline{R}=B$, $\overline{B}=R$.

\begin{theorem}
  Suppose that for each $j<n$, $\tau^j_0,\tau^j_1,\ldots,\tau^j_{m_j}$ is an $n+1$-train.  Let $c:[\mathbb{N}]^2\rightarrow\{R,B\}$ be given.  Then there is a coloring $c^*:\bigcup\tau^j_i\rightarrow\{R,B\}$ such that for each $\tau^j_i$ on which $c\upharpoonright[\tau^j_i]^2=\iota$ homogeneously, there is an $a\in\tau_i^j$ with $c^*(a)=\overline{\iota}$.  
\end{theorem}
\begin{proof}
  We define an ordering $\prec$ on the sets $\tau^j_i$: we say $\tau^j_i\prec\tau^{j'}_{i'}$ if, taking $k< n+1$ \emph{largest} such that $\tau^j_i(k)\neq \tau^{j'}_{i'}(k)$, we have $\tau^j_i(k)>\tau^{j'}_{i'}(k)$.  (This is the opposite of the reverse lexicographic order, which we choose not to refer to as the reverse reverse lexicographic order.)  Note that for a fixed $j$ and $i'<i$, we have that $\tau^j_i(n)>\tau^j_{i'}(n)$, and so $\tau^j_i\prec\tau^j_{i'}$.

For each $r$, let $j_r,i_r$ be such that $\tau^{j_r}_{i_r}$ is the $r$-th element in this ordering.

We define the coloring $c^*$ in stages, considering $\tau^{j_r}_{i_r}$ at the $r$-th stage.  We let $c_0=\emptyset$.  At stage $r$, we meet the $r$th requirement: that if $c\upharpoonright[\tau^{j_r}_{i_r}]^2=\iota$ homogeneously, then there is at least one $a\in\tau^{j_r}_{i_r}$ such that $c^*(a)=\overline{\iota}$. 
 Suppose we have constructed a partial function $c_r^*$ and, for some set of $s<r$, chosen values $a_s$ so that:
\begin{itemize}
\item at stage $r$, $c_r^*$ is only defined on $a_s$ for $s<r$,
\item if $a_s$ is defined then $a_s=\tau^{j_s}_{i_s}(k)$ for some $k>0$, $c\upharpoonright[\tau^{j_s}_{i_s}]^2=\iota_s$ homogeneously, and $c_r^*(a_s)=\overline{\iota}_s$,
\item for each $s<r$, if $c\upharpoonright[\tau^{j_s}_{i_s}]^2=\iota_s$ homogeneously then there is an $s'\leq s$ so that $a_{s'}\in\tau^{j_s}_{i_s}$ with $c_r^*(a_{s'})=\overline{\iota}_s$, and
\item if $s'<s<r$, $a_{s'},a_s$ are both defined, and $a_{s'}\in\tau^{j_s}_{i_s}$ then $c_r^*(a_s)\neq c_r^*(a_{s'})$.
\end{itemize}

The first clause asserts that we meet our requirements in order. The second asserts that if we acted at stage $s$, then we acted because there was a requirement to meet, and we acted to meet that requirement. It furthermore asserts that we did not act with the smallest element of $\tau^{j_s}_{i_s}$. The third clause asserts that each earlier requirement has been met. The final clause asserts that if a requirement was already met, then we did not act again to meet it. 


We make the following crucial observation: suppose $s'<s\leq r$, $j_{s'}=j_s$, $a_{s'}\in\tau^{j_s}_{i_s}$, and $c\upharpoonright[\tau^{j_s}_{i_s}]^2=\iota_s$ homogeneously.  Then $c_r^*(a_{s'})=\overline{\iota}_s$.  (This implies that $a_s$ is not defined.)
To see this, set $j=j_{s'}=j_s$ and observe that every $a\in\tau^{j}_{i_{s'}}\setminus\tau^{j}_{i_s}$ must have $a>\tau^{j}_{i_s}$.  Since $\tau^{j}_{i_{s'}}(0)<a_{s'}$, we must have $\tau^{j}_{i_{s'}}(0)\in\tau^{j}_{i_s}$, and therefore $c(\tau^{j}_{i_{s'}}(0),a_{s'})=\iota$.  Therefore $c_r^*(a_{s'})=\overline{\iota}$.

We now attempt to construct $c_{r+1}^*\supseteq c_r^*$.    If $c\upharpoonright[\tau^{j_r}_{i_r}]^2$ is not homogeneous, we have no commitment regarding $c^*\upharpoonright \tau^{j_r}_{i_r}$, so set $c^*_{r+1}=c_r^*$.  Suppose $c\upharpoonright[\tau^{j_r}_{i_r}]^2=\iota$ homogeneously.  If there is an $s<r$ such that $a_s$ is defined, $a_s\in\tau^{j_r}_{i_r}$, and $c^*(a_s)=\overline{\iota}$, then again we may set $c^*_{r+1}=c_r^*$.

So suppose there is no such $a_s$.  By the observation, if $s<r$, $j_s=j_r$, and $a_s$ is defined, we have $a_s\not\in\tau^{j_r}_{i_r}$.

If $s'<s<r$ with $j_{s'}=j_{s}$ and $a_{s'},a_{s}$ both defined, the observation implies that $a_{s'}\not\in\tau^{j_s}_{i_s}$.  Therefore $a_{s'}>\tau^{j_s}_{i_s}$.  But $\tau^{j_s}_{i_s}(n)\geq\tau^{j_r}_{i_r}(n)$, so $a_{s'}\not\in\tau^{j_r}_{i_r}$.

So $c^*_r$ is defined on at most $n-1$ elements of $\tau^{j_r}_{i_r}$---at most one for each $j<n$ other than $j_r$.  In particular, there are at least two elements in $\tau^{j_r}_{i_r}$ on which $c^*_r$ is undefined; taking the larger to be $a_r$, we set $c^*_r(a_r)=\overline{\iota}$, and we have $a_r=\tau^{j_r}_{i_r}(k)$ for some $k>0$.

We define $c^*$ to be any extension of $\bigcup_r c^*_r$ to a total function on $\bigcup\tau^j_i$.
\end{proof}

\section{Construction}

\begin{theorem}
  Fix finitely many $\Pi_2$ functionals given by formulas $\forall x\exists y R_0(z,x,y,c)$, $\ldots$, $\forall x\exists y R_{n-1}(z,x,y,c)$ depending on a coloring $c$.  There is a computable $c$ so that for each $j< n$, the set
\[S_j=\{z\mid\forall x\exists y R_j(z,x,y,c)\}\]
fails to be an infinite homogeneous set for $c$.
\end{theorem}
\begin{proof}
  We describe how, for a given $b$, we define $c(a,b)$ for all $a<b$.  Fix the value $b$ and suppose we have defined $c(a,a')$ for all $a<a'<b$.  For each $j<n$ we define an $n+1$-train by taking the set $\tau^j_i$ for $i\leq b$ to be the $n+1$ smallest elements $a<b$ such that $\forall x< i\exists y<b R_j(a,x,y,c)$ (where the computation is always true if $c$ is not yet sufficiently defined to interpret $R_j(a,x,y,c)$).  Let $c^*$ be given by the theorem above, and extend $c^*$ to be defined on all $a<b$ by defining it arbitrarily where it is not already defined.  Set $c(a,b)=c^*(a)$ for all $a<b$.

Suppose that for some $j<n$, the set $S_j$ is infinite.  Let $\tau^j$ be the $n+1$ smallest elements of $S_j$.  We claim that, for $b$ sufficiently large, there is always some $i$ so that $\tau^j_i=\tau^j$.  For every $a<\tau^j(n)$ such that $a\not\in\tau^j$, there is some $i$ such that $\exists x\leq i\forall y\neg R_j(a,x,y,c)$, so certainly for every $b$, if $i'\geq i$ and $a\in\tau^j_i$, either $a\in\tau^j$ or $a>\tau^j$.  Let $i$ be large enough to witness this bound for all $a<\tau^j(n)$.

For each $a\in\tau^j$ and each $x\leq i$, there is some $y$ such that $R_j(a,x,y,c)$.  If $b$ is big enough to bound these finitely many values of $y$, it must be the case that $\tau^j_i=\tau^j$.  Therefore for all sufficiently large $b$, $\tau^j_i=\tau^j$.

Since $S_j$ is infinite, let $b$ be some element of $S_j$ sufficiently large so that $\tau^j_i=\tau^j$.  If $c\upharpoonright[\tau^j]=\iota$ then there is some $a\in\tau^j$ with $c(a,b)=c^*(a)=\overline{\iota}$.  Therefore $S_j$ is not homogeneous.
\end{proof}

We can now prove our main theorem:
\begin{theorem}
  There is no partial computable $f$ such that for any $e$, if $e$ is the code of a total computable 2-coloring $c:[\mathbb{N}]^2\rightarrow 2$ of pairs, then $f(e)$ halts, producing the code for a c.e. set $W_{f(e)}$ enumerating a finite set $\{a_0,a_1,\dots,a_{n-1}\}$ of indices for $\Pi^0_2$ sets with the property that at least one of those $\Pi^0_2$ sets is an infinite homogeneous subset for $c$.
\end{theorem}
\begin{proof}
Let $f$ be a partial computable function such that for any $e$, if $e$ is the code of a total computable 2-coloring $c:[\mathbb{N}]^2\rightarrow 2$ of pairs, then $f(e)$ halts, producing the code for a c.e. set $W_{f(e)}$ enumerating a set $\{a_0,a_1,\dots,a_{n-1}\}$ of indices for $\Pi^0_2$ sets.


We define a coloring $c$ as follows.  Via the recursion theorem, we obtain the code for $c$, and begin evaluating $f(e)$.  If $f(e)$ has not halted after $b$ steps, we define $c(a,b)$ for $a<b$ arbitrarily.  If $f(e)$ has halted, then we begin enumerating $W_{f(e)}$. If $W_{f(e)}$ is empty after $b$ steps, we continue to define $c(a,b)$ for $a<b$ arbitrarily. Each time that $W_{f(e)}$ enumerates a new element, we continue the construction of $c$ as in the proof of the previous theorem, assuming that $W_{f(e)}$ will never enumerate any new elements.


If $W_{f(e)}$ is indeed finite, then at some point this assumption will be true, and we will be able to conclude that no $a_i$ is the code for an infinite homogeneous $\Pi^0_2$ subset for $c$.
Note that $c$ always produces a total computable 2-coloring, whether or not $f(e)$ halts, so the recursion theorem must produce a value $e$ on which $f(e)$ does halt.
\end{proof}


\begin{thebibliography}{9}
\bibitem{J} Josckusch, Carl, \emph{Ramsey's theorem and recursion theory}, J. Symbolic Logic 37 (1972), 268--280.

\end{thebibliography}
\end{document}